%% file: ex_article.tex
\documentclass[review]{siamart0216}

\input{ex_shared}

\ifpdf
\hypersetup{
  pdftitle={\TheTitle},
  pdfauthor={\TheAuthors}
}
\fi

\nolinenumbers
\begin{document}

\maketitle
\begin{abstract}
The fractional partial differential equations with a Caputo fractional derivative in time have attracted considerable interest in recent years. Considering the nonlocal nature of fractional derivative, finding fast solution technique to solve these efficient numerical schemes becomes instant and necessary. In this paper, we develop fast procedures for solving linear systems arising from discretization of ordinary and partial differential equations with Caputo fractional derivative. First, we consider a finite difference scheme to solve a two-sided fractional ordinary equation. A fast solution technique is proposed to accelerate Toeplitz matrix-vector multiplications arising from finite difference discretization. This fast solution technique is based on a fast Fourier transform and depends on the special structure of coefficient matrices, and it helps to reduce the computational work from $O(N^{3})$ required by traditional methods to $O(N$log$^{2}N)$ and the memory requirement from $O(N^{2})$ to $O(N)$ without using any lossy compression, where $N$ is the number of unknowns. The same idea is used for finite difference schemes to solve time fractional hyperbolic equations with different fractional order $\gamma$. We present a fast solution technique depending on the special structure of coefficient matrices by rearranging the order of unknowns. It helps to reduce the computational work from $O(MN^2)$ required by traditional methods to $O(MN$log$^{2}N)$ and the memory requirement from $O(NM)$ to $O(N)$ without using any lossy compression, where $N=\tau^{-1}$ and $\tau$ is the size of time step, $M=h^{-1}$ and $h$ is the size of space step. Importantly, a fast method is employed to solve the classical time fractional diffusion equation with a lower coast at $O(MN$log$^2N)$, where the direct method requires an overall computational complexity of $O(MN^2)$. Moreover, the applicability and accuracy of the scheme are demonstrated by numerical experiments to support our theoretical analysis.
\end{abstract}

\begin{keywords}
Fast procedures, Finite difference methods, Time fractional differential equations, Topelitz matrix, Fast Fourier transform.
\end{keywords}

\begin{AMS}
  65M06, 65M12, 65M15, 26A33
\end{AMS}

\section{Introduction}
Until recently, fractional calculus played a negligible role in physics, although its mathematics history is equally long. As fractional partial differential equations (FPDEs) appear frequently in diverse fields such as physics, biology, rheology, signal processing, systems identification and electrochemistry, they attract considerable interest and recently there has been significant theoretical development. Roughly speaking, the fractional models can be classified into two principal kinds: space-fractional differential equation and time-fractional one. In the last few decades the theory and numerical analysis of fractional differential equations have received increasing attention such as  finite element methods \cite{zhang2012finite,jiang2011high,li2010finite,zeng2013use}, finite difference methods \cite{sousa2012second,sousa2014explicit,sousa2009finite,huang2013two,liu2017crank,liu2016parallel}, finite volume (element) methods \cite{cheng2015eulerian,liu2014new1}, mixed finite element methods \cite{zhao2015two,liu2014mixed,liu2015h}, spectral methods \cite{lin2007finite,lin2011finite} and so on.

The basic definitions and properties of fractional derivatives including Riemann-Liouville derivative, the Caputo derivative, and some other fractional derivatives are introduced and relative important properties are studied in \cite{li2015numerical}. In this article, the time fractional derivative operator based on Caputo's definition is given by
\begin{flalign}\label{e1}
&_0^CD_t^{\gamma}g(t)=\frac{1}{\Gamma(n-\gamma)}\int_0^t(t-s)^{n-1-\gamma}g^{(n)}(s)ds,\quad n-1<\gamma\leq n.
\end{flalign}
From later on, for simplicity, we use $_0D_t^{\gamma}g(t)$  in stead of $_0^CD_t^{\gamma}g(t)$.

In this work, we focus on the fractional cases $0<\gamma<2$. This nonlocal definition is the limiting equation that
governs continuous time random walks with heavy tailed random waiting times. Numerical methods for solving fractional differential equations have been considered by many authors and we mention here a few key contributions. Chen et al. \cite{chen2007fourier} gave a Fourier method for the fractional diffusion equation describing sub-diffusion. Liu et al.\cite{liu2014rbf} developed a radial basis functions(RBFs) meshless approach for modeling a fractal mobile/immobile transport model. The authors of \cite{keskin2011approximate} proposed a new method based on the Taylor collocation method and obtained the approximate solution of linear fractional differential equation with variable coefficients. The finite difference schemes are used by Ashyralyev and Cakir\cite{ashyralyev2012numerical} for solving one-dimensional fractional parabolic partial differential equations. They \cite{ashyralyev2013fdm} also gave the FDM for fractional parabolic equations with the Neumann condition.

Due to the nonlocal nature of fractional differential operators, the numerical schemes of FPDEs give rise to dense stiffness matrices and/or long tails in time or a combination of both. This is deemed computationally very expensive in terms of computational complexity and memory requirement. Many articles consider fast conjugate gradient methods based on fast Fourier transform to solve space fractional equations. For example, Wang and Basu \cite{wang2012fast} presented a fast finite difference method for two-dimensional space-fractional diffusion equations. A fast characteristic finite difference method for fractional advection--diffusion equations was considered in \cite{wang2011fast}. Chen et al. \cite{chen2015fast} provided a fast semi-implicit difference method for a nonlinear two-sided space-fractional diffusion equation with variable diffusivity coefficients. For time fractional equations, unless we extract the Toeplitz structure for stiffness matrix, we  can not use  fast Fourier transform to speed up the evaluation. However, Ke et al. \cite{ke2015fast} studied a fast direct method for block triangular Toeplitz-like with tri-diagonal block systems for time-fractional partial differential equations. They reduce the computational work from $O(N^{3})$ required by traditional methods to $O(MN$log$^{2}M)$ and the memory requirement from $O(N^{2})$ to $O(MN)$, where $M$ is the number of blocks in the system and $N$ is the size of each block. Jiang et al. \cite{jiang2015fast} presented a fast evaluation scheme of the Caputo derivative to solve the fractional diffusion equations. This method requires $O(N_sN_{exp})$ storage and $O(N_sN_TN_{exp})$ work with $N_s$ the total number of points in space, $N_T$ the total number of time steps and $N_{exp}$ the number of exponentials. Fu and Wang \cite{wang2017preconditioned} developed a fast space-time finite difference method for space-time fractional diffusion equations by fully utilizing the mathematical structure of the scheme. In addition, their method has approximately linear computational complexity, i.e., has a computational cost of $O(MN$log$(MN))$ per Krylov subspace iteration. Our goal is to propose a fast direct method to solve both ordinary and partial fractional differential equations based on Caputo fractional derivative. We first consider a finite difference scheme to solve a two-sided fractional ordinary differential equation. Furthermore, we present a fast solution technique to accelerate Toeplitz matrix-vector multiplications arising from the finite difference discretization. This fast solution technique is based on a fast Fourier transform and depends on the special structure of coefficient matrices, and it helps to reduce the computational work from $O(N^{3})$ required by traditional methods to $O(N$log$^{2}N)$ and the memory requirement from $O(N^{2})$ to $O(N)$ without using any lossy compression, where $N$ is the number of unknowns. For time fractional partial differential hyperbolic equations with different fractional order $\gamma$, two finite difference schemes are considered. We present a fast solution technique depending on the special structure of coefficient matrices by rearranging the order of unknowns in space and time directions. It helps to reduce the computational work from $O(MN^2)$ required by traditional methods to $O(NM$log$^{2}N)$ and the memory requirement from $O(NM)$ to $O(N)$ without using any lossy compression, where $N=\tau^{-1}$ and $\tau$ is the size of time step, $M=h^{-1}$ and $h$ is the size of space step. Importantly, a fast method is employed to solve the classical time fractional diffusion equation with a lower cost at $O(MN$log$^2N)$, where the direct method requires an overall computational complexity of $O(MN^2)$.

This paper is organized as follows. In \cref{sec:fast-ordinary}, we analyze the structures of stiffness matrix for two-sided fractional ordinary differential equations and present a fast finite difference scheme. In \cref{sec:partial}, we introduce  fast procedures for finite difference method for several different kinds of time fractional partial differential equations with the Caputo fractional derivative . Then, in \cref{sec:numerical}, some numerical experiments for the finite difference discretization are carried out.

\section{A fast procedure of finite difference scheme for two-sided fractional ordinary differential equation}
\label{sec:fast-ordinary}
First, we consider the following two-sided fractional ordinary differential equation involving Caputo operators with general boundary conditions \cite{hernandez2016solution}:
\begin{equation}\label{e2}
_0D_t^{\gamma}u(t)+  _tD_1^{\gamma}u(t)+u(t)=f(t),\quad t\in(0,T),
\end{equation}
subject to the following conditions:
\begin{equation*}
u(0)=u(T)=0.
\end{equation*}

In \cite{hernandez2016solution}, the author provides well-posedness results and stochastic representations for the solutions to two-sided fractional ordinary differential equations involving both the right- and the left-sided generalized operators of Caputo type. The objective of this section is to consider the finite difference method for equations \cref{e2}. First, for the convenience of theoretical analysis, we introduce the following lemma:
\begin{lemma}\label{le1}
Denote
\begin{equation*}
\aligned
G_m=(m+1)^{1-\gamma}-m^{1-\gamma},&\quad m\geq0,\quad 0<\gamma<1,
\endaligned
\end{equation*}
then, we have the ineuqality
\begin{equation}\label{e3}
\aligned
&0<G_1\leq G_2\leq \cdots\leq G_N.\\
\endaligned
\end{equation}
\end{lemma}

Define $\tau=T/N$ and $\Omega_\tau=\{t_n|~t_n=i\tau,~0\leq i\leq N\}$ to be a uniform mesh of interval $(0,T)$. The values of the function $u$ at the grid points are denoted as $u_j=u(t_j)$. We also use $u_h(t_j)=u_j$ for gird function $u_h$ if no confusion occurs.

Define
\begin{flalign*}
&\|u_h\|_\infty=\max\limits_{1\leq i\leq N-1}|u_i|,\quad \|u_h\|^2=\sum\limits_{i=1}^{N-1}\tau u_i^2.
\end{flalign*}

\begin{lemma}\label{le2}
Define $\mu=\frac{\tau^{-\gamma}}{\Gamma(2-\gamma)}$ and suppose that $u\in C^2[0,T]$, then, the left-side fractional Caputo derivative $_0D_t^{\gamma}u(t)$ for $0<\gamma<1$ at $t=t_{n}$ is discretized by
\begin{equation}\label{e4}
_0D_t^{\gamma}u(t_n)=\mu\left[G_0u_n-\sum\limits_{k=1}^{n-1}(G_{n-k-1}-G_{n-k})u_k+G_nu_0\right]+O(\tau^{2-\gamma}),
\end{equation}
and the right-side fractional Caputo derivative $_tD_1^{\gamma}u(t)$ for $0<\gamma<1$ at $t=t_{n}$ is discretized by
\begin{equation}\label{e5}
\aligned
_tD_1^{\gamma}u(t_n)=
&\mu\left[G_0u_n-\sum\limits_{m=n+1}^{N-1}(G_{m-n-1}-G_{m-n})u_m+G_{N-n-1}u_N\right]+O(\tau^{2-\gamma}).
\endaligned
\end{equation}
\end{lemma}
\begin{proof}
The discrete formulation \cref{e4} has been proved in many articles such as \cite{lin2007finite}. Next, we will only prove the equation \cref{e5}. Firstly, we give the definition of right-side Caputo derivative \cite{hernandez2016solution}:
\begin{flalign}\label{e6}
&_tD_1^{\gamma}u(t)=-\frac{1}{\Gamma(1-\gamma)}\int_t^1u'(s)(s-t)^{-\gamma}ds,\quad 0<\gamma<1.
\end{flalign}
Then, a discrete approximation of the Caputo derivative \cref{e6} at $t_n$ can be obtained by the following approximation: $\forall$ $1\leq n\leq N-1$,
\begin{equation*}
\aligned
_tD_1^{\gamma}u(t_n)
&=-\frac{1}{\Gamma(1-\gamma)}\int_{t_n}^1u'(s)(s-t_n)^{-\gamma}ds\\
&=-\frac{1}{\Gamma(1-\gamma)}\sum\limits^{N-1}_{m=n}\int_{m\tau}^{(m+1)\tau}\frac{u'(s)}{(s-n\tau)^{\gamma}}ds\\
&=\displaystyle-\frac{1}{\Gamma(1-\gamma)}\sum\limits^{N-1}_{m=n}\int_{m\tau}^{(m+1)\tau}\left[\frac{u_{m+1}-u_{m}}{\tau}\right](s-n\tau)^{-\gamma}ds+R_{\tau}^n\\
&=\displaystyle-\frac{1}{\Gamma(2-\gamma)}\frac{1}{\tau^\gamma}\sum\limits^{N-1}_{m=n}\left(u_{m+1}-u_{m}\right)\left[(m-n+1)^{1-\gamma}-(m-n)^{1-\gamma}\right]+R_{\tau}^n\\
&=\displaystyle\frac{\tau^{-\gamma}}{\Gamma(2-\gamma)}\left[G_0u_n-\sum\limits_{m=n+1}^{N-1}(G_{m-n-1}-G_{m-n})u_m+G_{N-n-1}u_N\right]+R_{\tau}^n,
\endaligned
\end{equation*}
where
\begin{equation*}
\aligned
R_{\tau}^n
&\leq\displaystyle-c_u\left[\sum\limits^{N-1}_{m=n}\int_{m\tau}^{(m+1)\tau}\left[(2m+1)\tau-2s\right](s-n\tau)^{-\gamma}ds+O(\tau^2)\right],
\endaligned
\end{equation*}
where $c_u$ is a constant depending only on $u$. For $R_{\tau}^n$, it is easy to obtain
\begin{equation*}
\aligned
&-\sum\limits_{m=n}^{N-1}\int_{m\tau}^{(m+1)\tau}\left[(2m+1)\tau-2s\right](s-n\tau)^{-\gamma}ds\\
&=\displaystyle-\frac{\tau^{2-\gamma}}{1-\gamma}\sum\limits_{m=n}^{N-1}\left[(2m+1)G_{m-n}-2(m+1)(m-n+1)^{1-\gamma}+2m(m-n)^{1-\gamma}\right]\\
&\quad-\displaystyle\frac{2\tau^{2-\gamma}}{(1-\gamma)(2-\gamma)}\sum\limits_{m=n}^{N-1}\left[(m-n+1)^{2-\gamma}-(m-n)^{2-\gamma}\right]\\
&=\displaystyle\frac{h^{2-\gamma}}{1-\gamma}\left(\mathcal{S}_n-\frac{2}{2-\gamma}(N-n)^{2-\gamma}\right),
\endaligned
\end{equation*}
where
\begin{equation*}
\aligned
\mathcal{S}_n=(N-n)^{1-\gamma}+2[(N-n-1)^{1-\gamma}+(N-n-2)^{1-\gamma}+\cdots+2^{1-\gamma}+1^{1-\gamma}].
\endaligned
\end{equation*}
From \cite{lin2007finite}, it is not difficult to obtain
\begin{equation*}
\aligned
\left|\mathcal{S}_n-\frac{1}{2-\gamma}(N-n)^{2-\gamma}\right|\leq C,
\endaligned
\end{equation*}
where $C$ is a constant independent of $\gamma$ and $n$. It means that $|R_{\tau}^n|\leq C\tau^{2-\gamma}$. The proof is completed.
\end{proof}

Then, we get the following finite difference scheme:

\textbf{Scheme 1:}
\begin{equation}\label{e7}
\aligned
&\mu\left[2G_0u_n-\sum\limits_{k=1,k\neq n}^{N-1}(G_{|n-k-1|}-G_{|n-k|})u_k+G_nu_0+G_{N-n-1}u_N\right]\\
&\quad+u_n=f_n,\quad 1\leq n\leq N-1.
\endaligned
\end{equation}

From lemma \ref{le2} and \cite{liu2017fast}, it is easy to get the compatible condition
\begin{equation*}
\aligned
\lim\limits_{\tau\rightarrow0}\|R_h(u)\|=0.
\endaligned
\end{equation*}
Furthermore, suppose that the right function $f(t)$ is smooth enough, then we know that the finite difference equation is stable and the finite difference solution  $u_h$ is convergent and the convergence rate is $(2-\alpha)$.

\subsection{Solvability of the finite difference scheme}
\begin{theorem}\label{the1}
The finite difference scheme \cref{e7} is uniquely solvable and the stiffness matrix is a symmetric Toeplitz matrix.
\end{theorem}
\begin{proof}
 Noting that $\mu=\frac{\tau^{-\gamma}}{\Gamma(2-\gamma)}$ and $u_0=u_N=0$, then the discretized scheme \cref{e7} can be rewritten as
\begin{equation}\label{e8}
\aligned
&(1+2\mu)u_n-\mu\sum\limits_{k=1}^{n-1}(G_{n-k-1}-G_{n-k})u_k-\mu\sum\limits_{k=n+1}^{N-1}(G_{k-n-1}-G_{k-n})u_k=f_n.
\endaligned
\end{equation}

Furthermore, the discretized system for finite difference method can also be expressed in the following matrix form
\begin{flalign}\label{e9}
\begin{split}
\hspace{10mm}%
&Au=f,
\end{split}&
\end{flalign}
where we denote
\begin{equation*}
\aligned
&u=(u_1,u_2,\ldots,u_{N-2},u_{N-1})^T,\\
&f=(f_1,f_2,\ldots,f_{N-2},f_{N-1})^T.
\endaligned
\end{equation*}

Define $W_k=G_{k}-G_{k-1}$, then, the coefficient matrix $A$ can be expressed as
\begin{flalign*}
A=\left(
\begin{array}{cccccc}
1+2\mu&\mu W_1&\mu W_2&\cdots&\mu W_{N-2}\\
\mu W_1&1+2\mu&\mu W_1&\cdots&\mu W_{N-3}\\
\mu W_2&\mu W_1&1+2\mu&\ddots&\vdots\\
\vdots&\vdots&\ddots&\ddots& \mu  W_1\\
\mu W_{N-2}&\mu W_{N-3}&\cdots&\mu W_1&1+2\mu
\end{array}
\right).
\end{flalign*}
Thus, we obtain $A_{i,j}=A_{i+1,j+1}$ which means $A$ is a Toeplitz matrix.

Noting that $G_{j+1}\leq G_j$ for $1\leq j\leq N-1$, we have
\begin{equation*}
\aligned
A_{i,i}-\sum\limits_{j=1,j\neq i}^{N-1}|A_{i,j}|
&=(1+2\mu)-\mu\sum\limits_{k=1}^{i-1}(G_{n-k-1}-G_{n-k})-\mu\sum\limits_{k=i+1}^{N-1}(G_{k-n-1}-G_{k-n})\\
&=(1+2\mu)-(2-G_{i-1}-G_{N-i-1})\mu\\
&=1+\mu G_{i-1}+\mu G_{N-i-1}\\
&\geq0, \quad 1\leq i\leq N-1.
\endaligned
\end{equation*}
By the Gerschgorin circle theorem, the stiffness matrix $A$ is invertible. This invertibility guarantees the solvability of the discretized scheme. This completes the proof.
\end{proof}
\subsection{A fast procedure for finite difference method}
One of the most popular iterative methods for solving real symmetric positive definite systems such as \cref{e9} is conjugate gradient method. The relating algorithm is as follows. Let $u_{0}$ be an initial guess. Then compute $r_{0}=f-Au_{0}$, $\omega_{1}=r_{0}$ and
\begin{flalign*}
\begin{split}
\hspace{10mm}%
&\kappa_{1}=r^{T}_{0}r_{0}/\omega^{T}_{1}A\omega_{1}\\
\hspace{10mm}%
&u_{1}=u_{0}+\kappa_{k}\omega_{1}\\
\hspace{10mm}%
&r_{1}=r_{0}-\kappa_{1}A\omega_{1}\\
for\ &k=2,3,\ldots\\
\hspace{10mm}%
&\kappa_{k}=r^{T}_{k-1}r_{k-1}/r^{T}_{k-2}r_{k-2}\\
\hspace{10mm}%
&\omega_{k}=r_{k-1}+\kappa_{k}\omega_{k-1}\\
\hspace{10mm}%
&\kappa_{k}=r^{T}_{k-1}r_{k-1}/\omega^{T}_{k}A\omega_{k}\\
\hspace{10mm}%
&u_{k}=u_{k-1}+\kappa_{k}\omega_{k}\\
\hspace{10mm}%
&r_{k}=r_{k-1}-\kappa_{k}A\omega_{k}\\
\hspace{10mm}%
&\text{Check\ for\ convergence,\ continue\ if\ necessary}\\
end&\\
&u=u_{k}.
\end{split}&
\end{flalign*}

To reduce the computational work and memory requirement, we need only to accelerate the matrix-vector multiplication $A\omega$ for any vector $\omega$ and store $A$ efficiently. The stiffness matrix $A$ can be embedded into a $(2N-2)\times(2N-2)$ circulate matrix $C$ as follows
\begin{flalign*}
&C=\left(
\begin{array}{cccccc}
A&B\\
B&A
\end{array}
\right),
\qquad B=\left(
\begin{array}{cccccc}
0&\mu W_{N-2}&\cdots&\mu W_{2}&\mu W_{1}\\
\mu W_{N-2}&0&\mu W_{N-2}&\cdots&\mu W_{2}\\
\vdots&\mu W_{N-2}&0&\ddots&\vdots\\
\mu W_{2}&\vdots&\ddots&\ddots&\mu W_{N-2}\\
\mu W_{1}&\mu W_{2}&\cdots&\mu W_{N-2}&0
\end{array}
\right).
\end{flalign*}
The circulate matrix $C$ has the following decomposition 
\begin{flalign}\label{e10}
&C=F^{-1}diag(Fc)F,
\end{flalign}
where $c$ is the first column vector of $C$ and $F$ is the $2(N-1)\times2(N-1)$ discrete Fourier transform matrix. Denote that $v=(\omega,\omega)^T$, then it is well known that the matrix-vector multiplication $Fv$ for $v\in\mathbb{R}^{2N-2}$ can be carried out in $O(N$log$N)$ operations via the fast Fourier transform. Equation (\ref{e10}) shows that $Cv$ can be evaluated in $O(N$log$N)$ operations. So, we know that $A\omega$ can be evaluated in $O(N$log$N)$ operations for any $\omega\in\mathbb{R}^{N-1}.$ The overall computation cost of the fast conjugate method is $O(N$log$^2N)$. What's more, we only need to store first column vector of $C$ instead of the whole values of matrix $A$ which makes that the memory requirement will reduce from $O(N^{2})$ to $O(N)$.
\section{A fast finite difference scheme for time fractional partial differential equation}
\label{sec:partial}
In this subsection, we present some finite difference schemes and analyse the structures of stiffness matrices for several classical partial fractional equations. For convenience of theoretical analysis, we now denote
\begin{equation*}
\aligned
&G_k=(k+1)^{1-\gamma}-k^{1-\gamma},\quad 0<\gamma<1,\\
&M_k=(k+1)^{2-\gamma}-k^{2-\gamma},\quad 1<\gamma<2,\\
\endaligned
\end{equation*}
and
\begin{equation*}
d_tu^n=\frac{u^n-u^{n-1}}{\tau}.
\end{equation*}

From \cite{lin2007finite} and \cite{zhao2015two}, it is not difficult to verify that, for $\tau\rightarrow0$,
\begin{equation*}
1=G_0>G_1>G_2>\cdots>G_n>\cdots\rightarrow\tau^{\gamma}\rightarrow0,
\end{equation*}
and
\begin{equation*}
1=M_0>M_1>M_2>\cdots>M_n>\cdots\rightarrow\tau^{\gamma-1}\rightarrow0.
\end{equation*}

First, we consider the following fractional partial differential hyperbolic equation
\begin{equation}\label{e11}
_0^CD_t^{\gamma}u(x,t)+a\frac{\partial u(x,t)}{\partial x}=f(x,t),\quad x\in(0,L),\quad t\in(0,T],\quad 0<\gamma<1,
\end{equation}
subject to the initial condition:
\begin{equation*}
u(x,0)=u_0(x),\quad x\in[0,L],
\end{equation*}
with the boundary conditions
\begin{equation*}
\aligned
&u(0,t)=0,\quad t\in[0,T], \quad a>0,\\
&u(L,t)=0,\quad t\in[0,T], \quad a<0.
\endaligned
\end{equation*}

Without loss of generality, we set $a=1$. For above partial differential equation, we have the following discrete formula:

\textbf{Scheme 2:} Suppose $u\in C^{2,2}_{x,t}([0,L]\times[0,T])$, the time fractional Caputo derivative $_0^CD_t^{\gamma}u(x,t)$ for $0<\gamma<1$ at $(x_{i-\frac{1}{2}},t_{n})$ is discretized by \cite{lin2007finite}
\begin{equation}\label{e12}
\aligned
_0^CD_t^{\gamma}u(x_{i-\frac{1}{2}},t^n)=
&\frac{\tau^{-\gamma}}{2\Gamma(2-\gamma)}\left[G_0u_i^n-\sum\limits_{k=1}^{n-1}(G_{n-k-1}-G_{n-k})u_i^k+G_nu_i^0\right]\\
&+\frac{\tau^{-\gamma}}{2\Gamma(2-\gamma)}\left[G_0u_{i-1}^n-\sum\limits_{k=1}^{n-1}(G_{n-k-1}-G_{n-k})u_{i-1}^k+G_nu_{i-1}^0\right]\\
&+O(\tau^{2-\gamma}+h^2).
\endaligned
\end{equation}

For $\displaystyle
\frac{\partial u}{\partial x}$, we have
\begin{align}\label{e13}
\frac{\partial u(x_{i-\frac{1}{2}},t^n)}{\partial x}=\frac{u_i^n-u_{i-1}^n}{h}+O(h^2),
\end{align}
Replacing the function $u_i^n$ with its numerical approximation $U_i^n$, we get the following difference scheme:
\begin{equation*}
\aligned
&\left(\frac{G_0\tau^{-\gamma}}{2\Gamma(2-\gamma)}+\frac{1}{h}\right)U_i^n-\frac{\tau^{-\gamma}}{2\Gamma(2-\gamma)}\sum\limits_{k=1}^{n-1}(G_{n-k-1}-G_{n-k})U_i^k\\
&=\left(\frac{1}{h}-\frac{G_0\tau^{-\gamma}}{2\Gamma(2-\gamma)}\right)U_{i-1}^n+\frac{\tau^{-\gamma}}{2\Gamma(2-\gamma)}\sum\limits_{k=1}^{n-1}(G_{n-k-1}-G_{n-k})U_{i-1}^k\\
&\quad-\frac{\tau^{-\gamma}}{2\Gamma(2-\gamma)}G_nU_{i-1}^0-\frac{\tau^{-\gamma}}{2\Gamma(2-\gamma)}G_nU_i^0+f^n_{i-\frac{1}{2}}.
\endaligned
\end{equation*}

We have the following two procedures to solve the above scheme. For convenience, we denote $c=\frac{\tau^{-\gamma}}{2\Gamma(2-\gamma)}$. First, we introduce the following direct solver:

Let
\begin{equation*}
\aligned
&U^n=(U_1^n,U_2^n,\ldots,U_{M-2}^n,U_{M-1}^{n})^T,
\endaligned
\end{equation*}
\begin{equation*}
\aligned
F^n=
&(f_{\frac{1}{2}}^n-cG_nU_1^0-cG_nU_0^0,f_{\frac{3}{2}}^n-cG_nU_2^0-cG_nU_1^0,\ldots,\\
&f_{M-\frac{3}{2}}^{n}-cG_{n}U_{M-2}^0-cG_{n}U_{M-3}^0,f_{M-\frac{1}{2}}^{n}-cG_nU_{M-1}^0--cG_{n}U_{M-2}^0)^T,
\endaligned
\end{equation*}
then we obtain the matrix form of the finite difference formulation as follows:
\begin{equation}\label{e122}
AU^n=c\sum\limits_{k=1}^{n-1}B(G_{n-k}-G_{n-k-1})U^k+F^n.
\end{equation}

Noting that $U^n=(U_1^n,U^n_2,\ldots,U_{M-1}^n)^T$, then the stiffness matrix $A$ becomes the following formulation:
\begin{flalign*}
A=\left(
\begin{array}{cccccc}
\frac{1}{h}+cG_0&0&\cdots&0&0\\
cG_0-\frac{1}{h}&\frac{1}{h}+cG_0&\cdots&0&0\\
0&cG_0-\frac{1}{h}&\ddots&\vdots&\vdots\\
\vdots&\vdots&\ddots&\ddots&0\\
0&0&\cdots&cG_0-\frac{1}{h}&\frac{1}{h}+cG_0
\end{array}
\right),
\end{flalign*}
and
\begin{flalign*}
B=\left(
\begin{array}{cccccc}
1&0&\cdots&0&0\\
1&1&\cdots&0&0\\
0&1&\ddots&\vdots&\vdots\\
\vdots&\vdots&\ddots&\ddots&0\\
0&0&\cdots&1&1
\end{array}
\right).
\end{flalign*}

The part of RHS $\sum\limits_{k=1}^{n-1}B(G_{n-k}-G_{n-k-1})U^k$ will spend much computational work and it will make the whole computational requirement be $O(N^2M)$. To save the calculational time, we develop a fast procedure by changing solving orders between time and space. The original computing order is a time order:
\begin{equation*}
\aligned
U^1\rightarrow U^2\rightarrow\cdots U^i\rightarrow\cdots U^N.
\endaligned
\end{equation*}
We need to compute all spatial nodal values in a temporal layer, then to compute all spatial nodal values in next temporal layer at next time step (see Figure \ref{fig1}). The new procedure will change to be a space order (see Figure \ref{fig2}):
\begin{equation*}
\aligned
U_1\rightarrow U_2\rightarrow\cdots U_j\rightarrow\cdots U_{M-1},
\endaligned
\end{equation*}
where $U_j=\left(U_j^1,U_j^2,\cdots,U_j^{N}\right)^T$, $\forall$ $(j=1,2,\cdots,M-1)$.
\begin{figure}[htp]
\centering
\includegraphics[width=10cm,height=8cm]{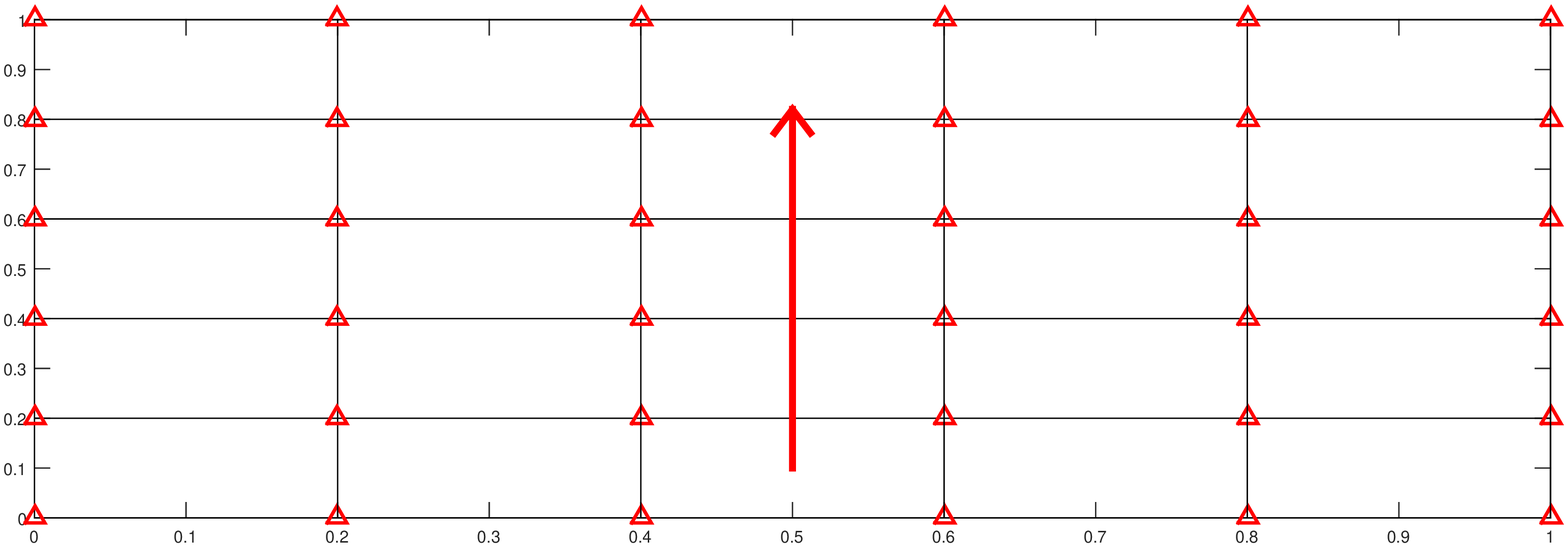}
\caption{The original order to compute nodal values.}\label{fig1}
\end{figure}
\begin{figure}[htp]
\centering
\includegraphics[width=10cm,height=8cm]{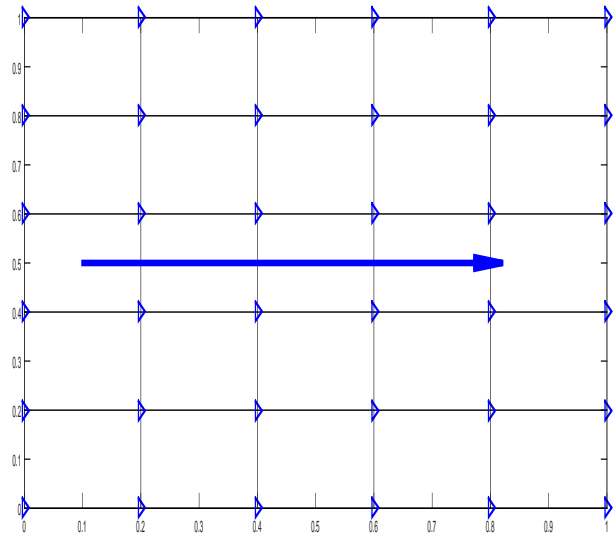}
\caption{The new order to compute nodal values.}\label{fig2}
\end{figure}

By using this new procedure, the matrix scheme (\ref{e122}) can be rewritten as
\begin{equation*}
\widetilde{A}U_i=BU_{i-1}+F_{i-\frac{1}{2}}.
\end{equation*}

By rearranging the order of unknowns, we use  $U_i=(U_i^1,U_i^2,\ldots,U_i^{N-1},U_i^{N})^T$ instead of $U^n=(U_1^n,U^n_2,\ldots,U_{M-1}^n)^T$, then the new stiffness matrix $\widetilde{A}$ can be expressed as
\begin{flalign*}
\widetilde{A}=\left(
\begin{array}{cccccc}
\displaystyle\frac{1}{h}+cG_0&0&\cdots&0&0\\
cG_1-cG_0&\displaystyle\frac{1}{h}+cG_0&0&\cdots&0\\
cG_2-cG_1&cG_1-cG_0&\displaystyle\frac{1}{h}+cG_0&\ddots&\vdots\\
\vdots&\vdots&\ddots&\ddots&0\\
cG_{N-1}-cG_{N-2}&cG_{N-2}-cG_{N-3}&\cdots&cG_1-cG_0&\displaystyle\frac{1}{h}+cG_0
\end{array}
\right),
\end{flalign*}
and
\begin{flalign*}
B=\left(
\begin{array}{cccccc}
\displaystyle\frac{1}{h}-cG_0&0&\cdots&0&0\\
cG_0-cG_1&\displaystyle\frac{1}{h}-cG_0&0&\cdots&0\\
cG_1-cG_2&cG_1-cG_2&\displaystyle\frac{1}{h}-cG_0&\ddots&\vdots\\
\vdots&\vdots&\ddots&\ddots&0\\
cG_{N-2}-cG_{N-1}&cG_{N-3}-cG_{N-2}&\cdots&cG_0-cG_1&\displaystyle\frac{1}{h}-cG_0
\end{array}
\right).
\end{flalign*}

By similar analysis as \cref{the1}, it is easy to obtain that $\widetilde{A}$ and $B$ are both asymmetric Toeplitz matrices. We can solve the scheme (\ref{e12}) by using the fast GMRES method which will be introduced in later section. We know that $\widetilde{A}X$ can be evaluated in $O(N$log$N)$ operations for any $X\in\mathbb{R}^{N}.$ The overall computation cost of the fast GMRES method will be reduced from original $O(N^2M)$ to $O(MN$log$^2N)$. What's more, it is not difficult to see that the memory requirement will reduce from $O(NM)$ to $O(N)$ without using any lossy compression.

Next, we consider the following fractional diffusion-wave equation
\begin{equation}\label{e14}
_0^CD_t^{\gamma}u(x,t)+a\frac{\partial u(x,t)}{\partial x}=f(x,t),\quad x\in(0,L),\quad t\in(0,T],\quad 1<\gamma<2,
\end{equation}
subject to the initial condition:
\begin{equation*}
u(x,0)=u_0(x),\quad \frac{\partial u(x,0)}{\partial t}=\phi(x),\quad x\in[0,L],
\end{equation*}
with the boundary conditions
\begin{equation*}
\aligned
&u(0,t)=0,\quad t\in[0,T], \quad a>0,\\
&u(L,t)=0,\quad t\in[0,T], \quad a<0,
\endaligned
\end{equation*}
without loss of generality, we define $a=1$.

\textbf{Scheme 3:} Suppose $u\in C^{2,3}_{x,t}([0,L]\times[0,T])$, the time fractional Caputo derivative $_0^CD_t^{\gamma}u(x,t)$ for $1<\gamma<2$ at $(x_{i},t_{n-\frac{1}{2}})$ is discretized by \cite{sun2006fully}
\begin{equation}\label{e15}
\aligned
_0^CD_t^{\gamma}u(x_{i},t_{n-\frac{1}{2}})=
&\frac{\tau^{1-\gamma}}{\Gamma(3-\gamma)}\left[M_0d_tu^{n-1/2}_i-\sum\limits_{k=1}^{n-1}\left(M_{n-k-1}-M_{n-k}\right)d_tu^{k-1/2}_i\right]\\
&-\frac{\tau^{1-\gamma}}{2\Gamma(3-\gamma)}M_{n-1}\phi_i+O(\tau^{3-\gamma}).
\endaligned
\end{equation}

Combining equation \cref{e13} with equation \cref{e15}, we obtain the following finite difference discretization formulation
\begin{equation}\label{e16}
\aligned
&\frac{\tau^{1-\gamma}}{\Gamma(3-\gamma)}\left[M_0d_tu^{n-1/2}_i-\sum\limits_{k=1}^{n-1}\left(M_{n-k-1}-M_{n-k}\right)d_tu^{k-1/2}_i\right]\\
&-\frac{\tau^{1-\gamma}}{2\Gamma(3-\gamma)}M_{n-1}\phi_i+\frac{u_i^n-u_{i-1}^n}{h}=f_{i}^{n-1/2}+O(\tau^{3-\alpha}+h).
\endaligned
\end{equation}

Define $c=\frac{\tau^{-\gamma}}{\Gamma(3-\gamma)}$, some simplification leads to
\begin{equation}\label{e17}
\aligned
&cM_0u_i^n+c(M_1-2M_0)u_i^{n-1}+c\sum\limits_{k=1}^{n-2}(M_{n-k-2}-2M_{n-k-1}+M_{n-k})u_i^k\\
&\quad+c(M_{n-2}-M_{n-1})u_i^0-c\tau M_{n-1}\phi_i+\frac{u_i^n-u_{i-1}^n}{h}=f_{i}^{n-1/2}+O(\tau^{3-\alpha}+h).
\endaligned
\end{equation}

\textbf{Direct Procedure}: let
\begin{equation*}
\aligned
&U^n=(U_1^n,U_2^n,\ldots,U_{M-2}^n,U_{M-1}^{n})^T,
\endaligned
\end{equation*}
\begin{flalign*}
F^{n-\frac{1}{2}}=\left(
\begin{array}{c}
f_{1}^{n-\frac{1}{2}}+c\tau M_{n-1}\phi_1-c(M_{n-2}-M_{n-1})u_1^0\\
f_{2}^{n-\frac{1}{2}}+c\tau M_{n-1}\phi_2-c(M_{n-2}-M_{n-1})u_2^0\\
\vdots\\
f_{M-1}^{n-\frac{1}{2}}+c\tau M_{n-1}\phi_{M-1}-c(M_{n-2}-M_{n-1})u_{M-1}^0
\end{array}
\right),
\end{flalign*}
then we obtain the matrix form of the finite difference formulation as follows:
\begin{equation*}
AU^n=c\sum\limits_{k=1}^{n-1}B(M_{n-k-2}-2M_{n-k-1}+M_{n-k})U^k+F^{n-\frac{1}{2}}.
\end{equation*}

Noting that $U^n=(U_1^n,U^n_2,\ldots,U_{M-1}^n)^T$,  the stiffness matrix $A$ can be expressed as:
\begin{flalign*}
A=\left(
\begin{array}{cccccc}
\frac{1}{h}+cM_0&0&\cdots&0&0\\
-\frac{1}{h}&\frac{1}{h}+cM_0&\cdots&0&0\\
0&-\frac{1}{2h}&\ddots&\vdots&\vdots\\
\vdots&\vdots&\ddots&\ddots&0\\
0&0&\cdots&-\frac{1}{h}&\frac{1}{h}+cM_0
\end{array}
\right),
\end{flalign*}
and for $1\leq k\leq n-2$,
\begin{flalign*}
B=\left(
\begin{array}{cccccc}
-1&0&\cdots&0&0\\
-1&-1&\cdots&0&0\\
0&-1&\ddots&\vdots&\vdots\\
\vdots&\vdots&\ddots&\ddots&0\\
0&0&\cdots&-1&-1
\end{array}
\right).
\end{flalign*}

The part of RHS $\sum\limits_{k=1}^{n-1}B(M_{n-k-2}-2M_{n-k-1}+M_{n-k})U^k$ will also spend much computational work and it will make the whole computational requirement be $O(N^2M)$. To save the calculational time, we use fast procedure by changing solving orders between time and space which has been introduced before.

\textbf{New Procedure}: By similar analysis as Scheme 3, we obtain the matrix form of the finite difference formulation as follows:
\begin{equation*}
\widetilde{A}U_i=\frac{1}{h}U_{i-1}+F_{i},
\end{equation*}
where the vector $U_i$ is
\begin{equation*}
\aligned
&U_i=(U_i^1,U_i^2,\ldots,U_i^{N-1},U_i^{N})^T,
\endaligned
\end{equation*}
and
\begin{flalign*}
F_{i}=\left(
\begin{array}{c}
f_{i}^{\frac{1}{2}}+c\tau M_0\phi_i\\
f_{i}^{\frac{3}{2}}+c\tau M_1\phi_i-c(M_{0}-M_{1})u_i^0\\
\vdots\\
f_{i}^{N-\frac{3}{2}}+c\tau M_{N-2}\phi_i-c(M_{N-3}-M_{N-2})u_i^0\\
f_{i}^{N-\frac{1}{2}}+c\tau M_{N-1}\phi_i-c(M_{N-2}-M_{N-1})u_i^0
\end{array}
\right),
\end{flalign*}
Denote $W_j=M_{j-2}-2M_{j-1}+M_{j}$, More importantly, the stiffness matrix $\widetilde{A}$ can be expressed as
\begin{flalign*}
\widetilde{A}=\left(
\begin{array}{cccccc}
\displaystyle\frac{1}{h}+cM_0&0&\cdots&0&0\\
cM_1-2cM_0&\displaystyle\frac{1}{h}+cM_0&0&\cdots&0\\
cW_2&cM_1-2cM_0&\displaystyle\frac{1}{h}+cM_0&\ddots&\vdots\\
\vdots&\vdots&\ddots&\ddots&0\\
cW_{N-1}&cW_{N-2}&\cdots&cM_1-2cM_0&\displaystyle\frac{1}{h}+cM_0
\end{array}
\right).
\end{flalign*}
By simple analysis, we can prove that $\widetilde{A}$ is a Toeplitz matrix.

Next, we consider the following time fractional diffusion equation
\begin{equation}\label{e18}
_0^CD_t^{\gamma}u(x,t)=\frac{\partial^2 u(x,t)}{\partial x^2}+f(x,t),\quad x\in(0,L),\quad t\in(0,T],\quad 0<\gamma<1,
\end{equation}
subject to the initial condition:
\begin{equation*}
u(x,0)=u_0(x),\quad x\in[0,L],
\end{equation*}
and the boundary conditions
\begin{equation*}
\aligned
&u(0,t)=u(L,t)=0,\quad t\in[0,T].
\endaligned
\end{equation*}

Many numerical approaches have been designed to solve above time fractional diffusion equation. Among them, we consider the classical implicit finite difference scheme, which has been proved to have unconditional stability and $L^2$ norm convergence. The classical finite difference scheme given by Lin and Xu \cite{lin2007finite} for the one-dimensional fractional diffusion equation has the following form:

\textbf{Scheme 4:}
\begin{equation}\label{e19}
\aligned
&\frac{\tau^{-\gamma}}{\Gamma(2-\gamma)}\left[G_0u_i^n-\sum\limits_{k=1}^{n-1}(G_{n-k-1}-G_{n-k})u_i^k+G_nu_i^0\right]\\
&=\frac{u_{i+1}^n-2u_i^n+u_{i-1}^n}{h^2}+f_i^n+O(\tau^{2-\gamma}+h^2).
\endaligned
\end{equation}

Then, we can rewrite the finite difference scheme to the following matric form:
\begin{equation}\label{e20}
\aligned
&AU^n=\frac{\tau^{-\gamma}}{\Gamma(2-\gamma)}\left[\sum\limits_{k=1}^{n-1}(G_{n-k-1}-G_{n-k})U^k+G_nU^0\right]+F^n,
\endaligned
\end{equation}
where  $U^n=(U_1^n,U^n_2,\ldots,U_{M-1}^n)^T$, $F^n=(F_1^n,F^n_2,\ldots,F_{M-1}^n)^T$, and the stiffness matrix $A$ becomes the following formulation:
\begin{flalign*}
A=\left(
\begin{array}{ccccccc}
\frac{2}{h^2}+cC_0&-\frac{1}{h^2}&0&\cdots&0&0\\
-\frac{1}{h^2}&\frac{2}{h}+cG_0&-\frac{1}{h^2}&\cdots&0&0\\
0&-\frac{1}{h^2}&\frac{2}{h}+cG_0&-\frac{1}{h^2}&\vdots&\vdots\\
\vdots&\vdots&\ddots&\ddots&\ddots&0\\
0&0&\cdots&-\frac{1}{h^2}&\frac{1}{h^2}+cG_0&-\frac{1}{h^2}\\
0&0&0&\cdots&-\frac{1}{h^2}&\frac{1}{h^2}+cG_0
\end{array}
\right).
\end{flalign*}

The linear system (\ref{e20}) can be solved by the time-marching method. However, it needs an average of $O(MN)$ operations to obtain a RHS vector and the time-marching method for system (\ref{e20}) requires an overall computational complexity of $O(MN^2)$. Next, we will consider a fast method to solve the linear system (\ref{e20}). First, we define
\begin{equation*}
\aligned
&\widetilde{U}=(U_1,U_2,\ldots,U_{M-2},U_{M-1})^T,
\endaligned
\end{equation*}
where
\begin{equation*}
\aligned
&U_i=(U_i^1,U_i^2,\ldots,U_i^{N-1},U_i^{N}).
\endaligned
\end{equation*}

Then, the linear system (\ref{e20}) can be rewritten as
\begin{equation}\label{e21}
\aligned
&\widetilde{A}\widetilde{U}=\widetilde{F},
\endaligned
\end{equation}
where the stiffness matrix $\widetilde{A}$ can be written as the following block matrix
\begin{flalign*}
\widetilde{A}=\left(
\begin{array}{ccccccc}
\widehat{A}&\widehat{B}&0&\cdots&0&0\\
\widehat{B}&\widehat{A}&\widehat{B}&\cdots&0&0\\
0&\widehat{B}&\widehat{A}&\widehat{B}&\vdots&\vdots\\
\vdots&\vdots&\ddots&\ddots&\ddots&0\\
0&0&\cdots&\widehat{B}&\widehat{A}&\widehat{B}\\
0&0&0&\cdots&\widehat{B}&\widehat{A}
\end{array}
\right).
\end{flalign*}
By simple calculation, it is easy to obtain
\begin{flalign*}
\widehat{A}=\left(
\begin{array}{cccccc}
\displaystyle\frac{2}{h^2}+cG_0&0&\cdots&0&0\\
cG_1-cG_0&\displaystyle\frac{1}{h}+cG_0&0&\cdots&0\\
cG_2-cG_1&cG_1-cG_0&\displaystyle\frac{1}{h}+cG_0&\ddots&\vdots\\
\vdots&\vdots&\ddots&\ddots&0\\
cG_{N-1}-cG_{N-2}&cG_{N-2}-cG_{N-3}&\cdots&cG_1-cG_0&\displaystyle\frac{1}{h}+cG_0
\end{array}
\right),
\end{flalign*}
where $\widehat{A}$ is the block lower Toeplitz matrix. Here the block matrix $\widehat{B}$ is
\begin{flalign*}
\widehat{B}=\left(
\begin{array}{cccccc}
\displaystyle-\frac{1}{h^2}&0&0&\cdots&0\\
0&\displaystyle-\frac{1}{h^2}&0&\cdots&0\\
\vdots&\vdots&\ddots&\cdots&0\\
0&\cdots&\cdots&\displaystyle-\frac{1}{h^2}&0\\
0&\cdots&\cdots&0&\displaystyle-\frac{1}{h^2}
\end{array}
\right).
\end{flalign*}

By making use of block Toeplitz structure in (\ref{e20}), the proposed procedure can be employed to solve the equations with a lower cost at $O(MN$log$^2N)$ by using fast Fourier transforms.
\subsection{The fast finite difference method}
\label{sec:fast method}
For general real nonsymmetric linear system $Au=b$, we introduce the generalized minimum residual (GMRES) method \cite{jin2010preconditioning}. The GMRES method was proposed by Saad and Schultz in 1986, which is one of the most important Krylov subspace methods for real nonsymmetric linear system \cite{saad1986gmres}. Let $u_{0}$ be an initial guess. Then the following GMRES algorithm can be used to solve  $Au=b$.
\begin{flalign*}
\begin{split}
\hspace{10mm}%
&r_0=b-Ax_0,\quad \beta=\|r_0\|_2,\quad v_1=r_0/\beta\\
\hspace{10mm}%
&\text{for}~j=1:k\\
\hspace{10mm}%
&\qquad w_j=Av_j\\
\hspace{10mm}%
&\qquad \text{for}~ i=1:j\\
\hspace{10mm}%
&\qquad\qquad h_{ij}=w_j^Tv_i\\
\hspace{10mm}%
&\qquad\qquad w_j=w_j-h_{ij}v_i\\
\hspace{10mm}%
&\qquad \text{end}\\
\hspace{10mm}%
&\qquad h_{j+1,j}=\|w_j\|_2\quad \text{if}~h_{j+1,j}=0,\ \text{set}\ k=j\ \text{and\ go\ to\ (*)}\\
\hspace{10mm}%
&\qquad v_{j+1}=w_j/h_{j+1,j}\\
\hspace{10mm}%
&\text{end}\\
\hspace{10mm}%
&\omega_{k}=r^{T}_{k-1}r_{k-1}/d^{T}_{k}Ad_{k}\\
\hspace{10mm}%
&\text{(*)\ compute}\ y_k\ \text{the\ minimizer\ of} \|\beta e_1-\widetilde{H}_ky\|_2\\
\hspace{10mm}%
&u_k=u_0+Vy_k\\
\hspace{10mm}%
&u=u_k\\
\end{split}&
\end{flalign*}
where $\widetilde{H}_k=(h_{ij})_{1\leq i\leq k+1,1\leq j\leq k}$. The matrix $V=(v_1,v_2,\ldots,v_k)\in \mathbb{R}^{n\times k}$ with $k\leq n$ is a matrix with orthonormal columns. In order to avoid a large storage and computational cost for the orthogonalization, the GMRES method is usually restarted after each $m$ $(m\ll n)$ iteration steps (refer to the GMRES(m) method).

To reduce the computational work and memory requirement, we need only to accelerate the matrix-vector multiplication $Ad$ for any vector $d$ and store $A$ efficiently. Let $a_{j-i}$ denote the common entry in the $(j-i)$-th descending diagonal of $A$ from left to right. Namely, $A_{i,j}=a_{j-i}$, $\forall j\geq i.$

The stiffness matrix $A$ can be embedded into a $2N\times2N$ circulate matrix $C$ as follows \cite{a10,a9,a2}
\begin{flalign*}
\begin{split}
\hspace{10mm}%
&C=\left(
\begin{array}{cccccc}
A&B\\
B&A
\end{array}
\right),
\qquad B=\left(
\begin{array}{cccccc}
0&a_{N-1}&\cdots&a_{2}&a_{1}\\
a_{N-1}&0&a_{N-2}&\cdots&a_{2}\\
\vdots&a_{N-2}&0&\ddots&\vdots\\
a_{2}&\vdots&\ddots&\ddots&a_{N-1}\\
a_{1}&a_{2}&\cdots&a_{N-1}&0
\end{array}
\right).
\end{split}&
\end{flalign*}
The circulate matrix $C$ has the following decomposition \cite{a10,a2,aa2}
\begin{flalign}\label{be1}
\begin{split}
\hspace{10mm}%
&C=F^{-1}diag(Fc)F,
\end{split}&
\end{flalign}
where $c$ is the first column vector of $C$ and $F$ is the $2N\times2N$ discrete Fourier transform matrix. Denote that $w=(d,d)^T$, then it is well known that the matrix-vector multiplication $Fw$ for $w\in\mathbb{R}^{2N}$ can be carried out in $O(N$log$N)$ operations via the fast Fourier transform. Equation (\ref{be1}) shows that $Cw$ can be evaluated in $O(N$log$N)$ operations. So, we know that $Ad$ can be evaluated in $O(N$log$N)$ operations for any $d\in\mathbb{R}^{N}.$ The overall computation cost of the fast GMRES method is $O(MN$log$^2N)$.
\section{Numerical results}
\label{sec:numerical}
In this section, some computational experiments have been carried out. Here we use Gaussian elimination (Gauss), the conjugate gradient (CG) method, and the fast conjugate gradient (FCG) method to solve two-sided fractional ordinary differential equation and direct method and fast
method including fast generalized minimum residual method to solve time fractional partial differential equations for comparison. These methods were implemented in Matlab, and the numerical experiments were run on a 8-GB memory computer.

\textbf{Example 1}: we consider the following two-sided ordinary fractional differential equation with the real solution $u(t)=t(1-t)$:
\begin{flalign*}
\renewcommand{\arraystretch}{1.5}
  \left\{
   \begin{array}{l}
_0D_t^{\gamma}u(t)+_tD_1^{\gamma}u(t)+u(t)=f(t),\quad t\in(0,1),\\
u(0)=u(1)=0,
\end{array}\right.
\end{flalign*}
with forcing function
\begin{equation*}
\aligned
\displaystyle f(t)=\frac{t^{1-\gamma}}{\Gamma(2-\gamma)}-\frac{2t^{2-\gamma}}{\Gamma(3-\gamma)}+\frac{2t^{2-\gamma}1}{\Gamma(3-\gamma)}(\gamma t-2t^2-\gamma+2t)(1-t)^{-\gamma}+t(1-t).
\endaligned
\end{equation*}

In the tables and figures, we provide the $L^2(\Omega)$ norm of the error and the corresponding rates of convergence for a sequence of grid sizes. From \cref{tab:tab1}, one can see that Scheme 1 for the ordinary differential model \cref{e2} converge at the optimal rates  of $O(h^{2-\gamma})$. In \cref{tab:tab2}, we present the CPU time consumed by the Gauss, CG and FCG methods. We observe that the FCG solver results in a significantly less computational time, compared to the Gauss and CG solvers. Moreover, as the mesh size $h$ is reduced to $h=2^{-15}$, both Gauss and CG methods run out of memory, whereas the FCG method solves the problem still consuming relatively little computational time.

\begin{table}[h!b!p!]
\small
\renewcommand{\arraystretch}{1.1}
\centering
\caption{\small The $L_2$ errors and convergence rates for Scheme 1 for different values of $\gamma$.}\label{tab:tab1}
\begin{tabular}{cccccccccc}
\hline
\multicolumn{1}{l}{$\tau$}& \multicolumn{2}{l}{$\gamma=0.1$}&&\multicolumn{2}{l}{$\gamma=0.5$}&&\multicolumn{2}{l}{$\gamma=0.9$}\\
\cline{2-3}\cline{5-6} \cline{8-9}
&\multicolumn{1}{l}{$L_2$ error} &\multicolumn{1}{r}{Rate} && \multicolumn{1}{l}{$L_2$ error} &\multicolumn{1}{r}{Rate} && \multicolumn{1}{l}{$L_2$ error} &\multicolumn{1}{r}{Rate} \\
\hline
\multicolumn{1}{l}{$1/2^5$}&7.8554e-5&-&&1.5571e-3&-&&1.9214e-2&-\\
\multicolumn{1}{l}{$1/2^6$}&2.3116e-5&\multicolumn{1}{r}{1.7648}&&5.7041e-4&\multicolumn{1}{r}{1.4488}&&9.6294e-3&\multicolumn{1}{r}{0.9966}\\
\multicolumn{1}{l}{$1/2^7$}&6.7000e-6&\multicolumn{1}{r}{1.7867}&&2.0606e-4&\multicolumn{1}{r}{1.4689}&&4.6751e-3&\multicolumn{1}{r}{1.0424}\\
\multicolumn{1}{l}{$1/2^8$}&1.9201e-6&\multicolumn{1}{r}{1.8030}&&7.3842e-5&\multicolumn{1}{r}{1.4806}&&2.2298e-3&\multicolumn{1}{r}{1.0681}\\
\multicolumn{1}{l}{$1/2^{9}$}&5.4550e-7&\multicolumn{1}{r}{1.8155}&&2.6334e-5&\multicolumn{1}{r}{1.4875}&&1.0530e-3&\multicolumn{1}{r}{1.0824}\\
\multicolumn{1}{l}{$1/2^{10}$}&1.5388e-7&\multicolumn{1}{r}{1.8258}&&9.3635e-6&\multicolumn{1}{r}{1.4918}&&4.9455e-4&\multicolumn{1}{r}{1.0903}\\
\multicolumn{1}{l}{$1/2^{11}$}&4.3158e-8&\multicolumn{1}{r}{1.8341}&&3.3230e-6&\multicolumn{1}{r}{1.4946}&&2.3155e-4&\multicolumn{1}{r}{1.0948}\\
\multicolumn{1}{l}{$1/2^{12}$}&1.2045e-8&\multicolumn{1}{r}{1.8412}&&1.1779e-6&\multicolumn{1}{r}{1.4963}&&1.0823e-4&\multicolumn{1}{r}{1.0972}\\
\hline
\end{tabular}
\end{table}

\begin{table}[htbp]
\small
\renewcommand{\arraystretch}{1.0}
\centering
   \begin{center}
\caption{The CPU time consumed by the Gauss, CG and FCG methods for different mesh sizes for the Scheme 1.}\label{tab:tab2}
\begin{tabular}{cccccccccc}\hline
\multicolumn{1}{l}{$\tau$}& \multicolumn{2}{c}{Gauss method}&\multicolumn{2}{c}{CG method}&\multicolumn{2}{c}{FCG method}\\
\cline{2-3}\cline{4-5} \cline{6-7}
&\multicolumn{1}{c}{CPU} &\multicolumn{1}{c}{$\#$ of iter.} & \multicolumn{1}{c}{CPU} &\multicolumn{1}{c}{$\#$ of iter.} & \multicolumn{1}{c}{CPU} &\multicolumn{1}{c}{$\#$ of iter.} \\
\hline
$2^{-9}$     &1.20 s             &-&0.21 s                &62  & 0.01 s                &62   \\
$2^{-10}$    &9.90 s             &-&1.27 s                &75  & 0.06 s                &75   \\
$2^{-11}$    &103  s             &-&5.58 s                &91  & 0.09 s                &91   \\
$2^{-12}$    &931  s             &-&23.8 s                &110 & 0.57 s                &110  \\
$2^{-13}$    &7328 s             &-&105  s                &133 & 1.29 s                &133  \\
$2^{-14}$    &$>$10 h            &-&515  s                &159 & 2.24 s                &159  \\
$2^{-15}$    &Out of memory      &-&Out of memory         &-   & 4.58 s                &192  \\
$2^{-16}$    &N/A                &-&N/A                   &-   & 9.81 s                &230  \\
$2^{-17}$    &N/A                &-&N/A                   &-   & 27.2 s                &276  \\
$2^{-18}$    &N/A                &-&N/A                   &-   & 80.9 s                &331  \\
$2^{-19}$    &N/A                &-&N/A                   &-   & 281  s                &415  \\
$2^{-20}$    &N/A                &-&N/A                   &-   & 451  s                &498  \\
$2^{-21}$    &N/A                &-&N/A                   &-   & 977  s                &599  \\
$2^{-22}$    &N/A                &-&N/A                   &-   & 2548 s                &741  \\\hline
\end{tabular}
\end{center}
\end{table}

\textbf{Example 2}: we consider the following partial fractional differential equation with the real solution $u(x,t)=t\sin(\pi x)$:
\begin{flalign*}
\renewcommand{\arraystretch}{1.5}
  \left\{
   \begin{array}{l}
\displaystyle_0^CD_t^{\gamma}u(x,t)+\frac{\partial u(x,t)}{\partial x}=f(x,t),\quad x\in(0,1),\ t\in(0,1],\quad 0<\gamma<1,\\
u(x,0)=0,\quad x\in[0,1],\\
u(0,t)=u(1,t)=0,\quad t\in[0,1],
\end{array}\right.
\end{flalign*}
with forcing function
\begin{equation*}
\aligned
\displaystyle f(x,t)=\frac{t^{1-\gamma}}{\Gamma(2-\gamma)}\sin(\pi x)+\pi t\cos(\pi x).
\endaligned
\end{equation*}

We check the errors and convergence rates for Scheme 2 by using Example 2. We fixed the temporal mesh size $\tau=1/2^{10}$ and refine the spatial mesh size $h$ from $\frac{1}{8}$ to $\frac{1}{256}$. \cref{tab:tab5} shows the numerical results for two different fractional values: $\gamma=0.1$ and $\gamma=0.9$. While \cref{tab:tab5} shows that the convergence order in space is $O(h^2)$ for both the direct method and our fast procedure. From \cref{tab:tab6}, we observe that the fast proedure is much faster than the direct method.
\begin{table}[h!b!p!]
\small
\renewcommand{\arraystretch}{1.1}
\centering
\caption{\small The $L_2$ errors and convergence rates for direct method and fast method with different values of $\gamma$.}\label{tab:tab5}
\begin{tabular}{cccccccccccc}
\hline
\multicolumn{1}{l}{$h$}&\multicolumn{2}{c}{$\gamma=0.1$}&&&&&\multicolumn{2}{c}{$\gamma=0.9$}\\
\cline{2-5}\cline{7-10}
&\multicolumn{2}{c}{Direct method}&\multicolumn{2}{c}{Fast method}&&\multicolumn{2}{c}{Direct method}&\multicolumn{2}{c}{Fast method}\\
\cline{2-5}\cline{7-10}
&\multicolumn{1}{l}{$L_2$ error}&\multicolumn{1}{r}{Rate} &\multicolumn{1}{l}{$L_2$ error}&\multicolumn{1}{r}{Rate}&& \multicolumn{1}{l}{$L_2$ error}&\multicolumn{1}{r}{Rate}  & \multicolumn{1}{l}{$L_2$ error} &\multicolumn{1}{r}{Rate} \\
\hline
\multicolumn{1}{l}{$1/2^3$}&7.8115e-2&-   &7.8113e-2&-   &&8.5609e-2&-   &8.5608e-2&-   \\
\multicolumn{1}{l}{$1/2^4$}&1.9620e-2&1.99&1.9619e-2&1.99&&2.1600e-2&1.98&2.1598e-2&1.98\\
\multicolumn{1}{l}{$1/2^5$}&4.9189e-3&1.99&4.8637e-3&2.01&&5.4270e-3&1.99&5.4222e-3&1.99\\
\multicolumn{1}{l}{$1/2^6$}&1.2316e-3&2.00&1.2039e-3&2.01&&1.3602e-3&2.00&1.3536e-3&2.00\\
\multicolumn{1}{l}{$1/2^7$}&3.0816e-4&2.00&2.9434e-4&2.03&&3.4051e-4&2.00&3.2967e-4&2.03\\
\multicolumn{1}{l}{$1/2^8$}&7.7072e-5&2.00&7.0183e-5&2.06&&8.5185e-5&2.00&7.3666e-5&2.16\\
\hline
\end{tabular}
\end{table}

\begin{table}[htbp]
\small
\renewcommand{\arraystretch}{1.1}
\centering
   \begin{center}
\caption{The CPU time consumed by direct and fast methods with different mesh sizes for Scheme 2 with $\gamma=0.9$.}\label{tab:tab6}
\begin{tabular}{|c|c|c|c|c|c|c|c|c|}
\hline
$h=\tau$ &$1/2^7$ &$1/2^8$ &$1/2^9$ &$1/2^{10}$ &$1/2^{11}$ &$1/2^{12}$&$1/2^{13}$\\
\hline
Direct method (CPU)&7.14s&57.7s&471s&3826s&29333s&$>$10h&$>$2d\\
\hline
Fast method (CPU)  &0.17s&0.43s&1.49s&5.56s&21.4s&399s&1191s\\
\hline
\end{tabular}
\end{center}
\end{table}

\textbf{Example 3}: we consider the following partial fractional differential equation with the real solution $u(x,t)=t^3x(1-x)$:
\begin{flalign*}
\renewcommand{\arraystretch}{1.5}
  \left\{
   \begin{array}{l}
\displaystyle_0^CD_t^{\gamma}u(x,t)+\frac{\partial u(x,t)}{\partial x}=f(x,t),\quad x\in(0,1),\ t\in(0,1],,\quad 1<\gamma<2,\\
u(x,0)=0,\frac{\partial u(x,0)}{\partial t}=0,\quad x\in[0,1],\\
u(0,t)=0,\quad t\in[0,1],
\end{array}\right.
\end{flalign*}
with forcing function
\begin{equation*}
\aligned
\displaystyle f(x,t)=\frac{6t^{3-\gamma}}{\Gamma(4-\gamma)}(x-x^2)+t^3(1-2x).
\endaligned
\end{equation*}
Then, we check the errors and convergence rates for Scheme 3 by using Example 3. We take $h=\tau$ from $\frac{1}{8}$ to $\frac{1}{256}$ to check the convergence rates. \cref{tab:tab7} shows the numerical results for two different fractional values: $\gamma=1.1$ and $\gamma=1.9$. While \cref{tab:tab7} shows that the convergence order in space is $O(h)$ for both the direct method and our fast method. From \cref{tab:tab8}, we observe that our fast procedure is much faster than the direct procedure.
\begin{table}[h!b!p!]
\small
\renewcommand{\arraystretch}{1.1}
\centering
\caption{\small The $L_2$ errors and convergence rates for direct method and fast method with different values of $\gamma$.}\label{tab:tab7}
\begin{tabular}{cccccccccccc}
\hline
\multicolumn{1}{l}{$h$}&\multicolumn{2}{c}{$\gamma=1.1$}&&&&&\multicolumn{2}{c}{$\gamma=1.9$}\\
\cline{2-5}\cline{7-10}
&\multicolumn{2}{c}{Direct method}&\multicolumn{2}{c}{Fast method}&&\multicolumn{2}{c}{Direct method}&\multicolumn{2}{c}{Fast method}\\
\cline{2-5}\cline{7-10}
&\multicolumn{1}{l}{$L_2$ error}&\multicolumn{1}{r}{Rate} &\multicolumn{1}{l}{$L_2$ error}&\multicolumn{1}{r}{Rate}&& \multicolumn{1}{l}{$L_2$ error}&\multicolumn{1}{r}{Rate}  & \multicolumn{1}{l}{$L_2$ error} &\multicolumn{1}{r}{Rate} \\
\hline
\multicolumn{1}{l}{$1/2^3$}&3.0960e-2&-   &3.0960e-2&-   &&1.7284e-2&-   &1.7484e-2&-   \\
\multicolumn{1}{l}{$1/2^4$}&1.6276e-2&0.93&1.6275e-2&0.93&&9.8789e-3&0.81&9.8789e-3&0.81\\
\multicolumn{1}{l}{$1/2^5$}&8.3729e-3&0.96&8.3729e-3&0.96&&5.0070e-3&0.98&5.0070e-3&0.98\\
\multicolumn{1}{l}{$1/2^6$}&4.2530e-3&0.98&4.2529e-3&0.98&&2.4417e-3&1.03&2.4418e-3&1.03\\
\multicolumn{1}{l}{$1/2^7$}&2.1445e-3&0.99&2.1444e-3&0.99&&1.1728e-3&1.05&1.1729e-3&1.05\\
\multicolumn{1}{l}{$1/2^8$}&1.0769e-3&0.99&1.0769e-3&0.99&&5.6093e-4&1.06&5.6694e-4&1.04\\
\hline
\end{tabular}
\end{table}
\begin{table}[htbp]
\small
\renewcommand{\arraystretch}{1.1}
\centering
   \begin{center}
\caption{The CPU time consumed by direct and fast methods with different mesh sizes for Scheme 3 with $\gamma=1.5$.}\label{tab:tab8}
\begin{tabular}{|c|c|c|c|c|c|c|c|c|}
\hline
$h=\tau$ &$1/2^6$ &$1/2^7$ &$1/2^8$ &$1/2^{9}$ &$1/2^{10}$ &$1/2^{11}$&$1/2^{12}$\\
\hline
Direct method (CPU)&0.62s&4.77s&38.0s&302s&2419s&$>$10h&$>$2d\\
\hline
Fast method (CPU)  &0.14s&0.23s&0.57s&2.01s&6.67s&27.4s&573s\\
\hline
\end{tabular}
\end{center}
\end{table}

\textbf{Example 4}: we consider the following time fractional diffusion equation with the real solution $u(x,t)=t^3\sin(\pi x)$:
\begin{flalign*}
\renewcommand{\arraystretch}{1.5}
  \left\{
   \begin{array}{l}
\displaystyle_0^CD_t^{\gamma}u(x,t)=\frac{\partial^2 u(x,t)}{\partial x^2}+f(x,t),\quad x\in(0,1),\ t\in(0,1],,\quad 0<\gamma<1,\\
u(x,0)=0,\quad x\in[0,1],\\
u(0,t)=0,\quad t\in[0,1],
\end{array}\right.
\end{flalign*}
with forcing function
\begin{equation*}
\aligned
\displaystyle f(x,t)=\frac{6t^{3-\gamma}}{\Gamma(4-\gamma)}\sin(\pi x)+\pi^2t^3\sin(\pi x).
\endaligned
\end{equation*}
Errors and convergence rates are considered for Scheme 4 by using Example 4. We take $h=\tau$ from $\frac{1}{8}$ to $\frac{1}{256}$ to check the convergence rates. \cref{tab:tab9} shows the numerical results for two different fractional values: $\gamma=0.1$ and $\gamma=0.9$. While \cref{tab:tab9} shows that the direct method and our fast procedure have the same convergence order. From \cref{tab:tab10}, we can also observe that our fast method is much faster than the direct method.
\begin{table}[h!b!p!]
\small
\renewcommand{\arraystretch}{1.1}
\centering
\caption{\small The $L_2$ errors and convergence rates for direct method and fast method with different values of $\gamma$ for Scheme 4.}\label{tab:tab9}
\begin{tabular}{cccccccccccc}
\hline
\multicolumn{1}{l}{$h=\tau$}&\multicolumn{2}{c}{$\gamma=0.1$}&&&&&\multicolumn{2}{c}{$\gamma=0.9$}\\
\cline{2-5}\cline{7-10}
&\multicolumn{2}{c}{Direct method}&\multicolumn{2}{c}{Fast method}&&\multicolumn{2}{c}{Direct method}&\multicolumn{2}{c}{Fast method}\\
\cline{2-5}\cline{7-10}
&\multicolumn{1}{l}{$L_2$ error}&\multicolumn{1}{r}{Rate} &\multicolumn{1}{l}{$L_2$ error}&\multicolumn{1}{r}{Rate}&& \multicolumn{1}{l}{$L_2$ error}&\multicolumn{1}{r}{Rate}  & \multicolumn{1}{l}{$L_2$ error} &\multicolumn{1}{r}{Rate} \\
\hline
\multicolumn{1}{l}{$1/2^3$}&8.4669e-3&-   &8.4669e-3&-   &&2.3623e-2&-   &2.3623e-2&-   \\
\multicolumn{1}{l}{$1/2^4$}&2.1203e-3&1.99&2.1203e-3&1.99&&9.5705e-3&1.30&9.5705e-3&1.30\\
\multicolumn{1}{l}{$1/2^5$}&5.3323e-4&1.99&5.3323e-4&1.99&&4.1202e-3&1.21&4.1202e-3&1.21\\
\multicolumn{1}{l}{$1/2^6$}&1.3428e-4&1.98&1.3425e-4&1.98&&1.8379e-3&1.16&1.8378e-3&1.16\\
\multicolumn{1}{l}{$1/2^7$}&3.3841e-5&1.98&3.3796e-5&1.99&&8.3651e-4&1.13&8.3140e-4&1.14\\
\multicolumn{1}{l}{$1/2^8$}&8.5330e-6&1.99&8.4517e-6&2.00&&3.8563e-4&1.12&3.5459e-4&1.22\\
\hline
\end{tabular}
\end{table}
\begin{table}[htbp]
\small
\renewcommand{\arraystretch}{1.1}
\centering
   \begin{center}
\caption{The CPU time consumed by direct and fast methods with different mesh sizes for Scheme 4 with $\gamma=0.1$.}\label{tab:tab10}
\begin{tabular}{|c|c|c|c|c|c|c|c|}
\hline
$h=\tau$ &$1/2^5$ &$1/2^6$ &$1/2^7$ &$1/2^{8}$ &$1/2^{9}$ &$1/2^{10}$\\
\hline
Direct method (CPU)&0.09s&0.73s&6.09s&57.0s&641s&9863s\\
\hline
Fast method (CPU)  &0.00s&0.07s&0.26s&2.27s&15.6s&263s\\
\hline
\end{tabular}
\end{center}
\end{table}

\section{Conclusions}
\label{sec:conclusions}
 For  time fractional equations,  unless we extract the Toeplitz structure for stiffness matrix, we  can not use  fast Fourier transform to speed up the evaluation. In this article, we study a fast procedure to solve both ordinary and partial fractional differential equations based on Caputo fractional derivative. For differential equations with Caputo fractional time derivative, we give several different finite difference schemes and analyse the stiffness matrices. We present a fast solution technique depending on the special structure of coefficient matrices by rearranging the order of unknowns in time and space directions. We observe that our fast procedure is much faster than the direct procedure.

\section*{Acknowledgments}
We would like to acknowledge the assistance of volunteers in putting together this example manuscript and supplement. This work was supported in part by the National Natural Science Foundation of China under Grants 91630207, 11471194 and 11571115, by the National Science Foundation under Grant DMS-1216923,by the OSD/ARO MURI Grant W911NF-15-1-0562, by the National Science and Technology Major Project of China under Grants 2011ZX05052 and 2011ZX05011-004, and by Shandong Provincial Natural Science Foundation, China under Grant ZR2011AM015, and by Taishan Scholars Program of Shandong Province of China.

\bibliographystyle{siamplain}
\bibliography{references}
\end{document}

%% file: ex_shared.tex
\usepackage{lipsum}
\usepackage{amsfonts}
\usepackage{graphicx}
\usepackage{epstopdf}
\usepackage{algorithmic}
\ifpdf
  \DeclareGraphicsExtensions{.eps,.pdf,.png,.jpg}
\else
  \DeclareGraphicsExtensions{.eps}
\fi

\newcommand{\TheTitle}{Fast procedures for Caputo fractional ordinary and partial differential equations}
\newcommand{\TheAuthors}{Zhengguang Liu, Aijie Cheng, Xiaoli Li, and Hong Wang}

\headers{Fast procedures for the Caputo fractional derivative}{\TheAuthors}

\title{{\TheTitle}\thanks{Received by editors January 17, 2018 }}

\author{
 Zhengguang Liu\thanks{School of Mathematics, Shandong University, Jinan, Shandong 250100,
 China.(\email{liuzhgsdu@yahoo.com}).}
 \and
 Aijie Cheng\thanks{Corresponding author. School of Mathematics, Shandong University, Jinan, Shandong 250100, China.(\email{aijie@sdu.edu.cn}).}
 \and
 Xiaoli Li\thanks{School of Mathematics, Shandong University, Jinan, Shandong 250100, China.(\email{xiaolisdu@163.com}).}
 \and
 Hong Wang\thanks{Department of Mathematics, University of South Carolina, Columbia, SC 29208, USA.(\email{hwang@math.sc.edu}).}
}

\usepackage{amsopn}


%% file: ex_article.bbl
\begin{thebibliography}{10}

\bibitem{ashyralyev2012numerical}
{\sc A.~Ashyralyev and Z.~Cakir}, {\em On the numerical solution of fractional
  parabolic partial differential equations with the dirichlet condition},
  Discrete Dynamics in Nature and Society, 2012 (2012),
  \url{http://dx.doi.org/10.1155/2012/696179}.

\bibitem{ashyralyev2013fdm}
{\sc A.~Ashyralyev and Z.~Cakir}, {\em Fdm for fractional parabolic equations
  with the neumann condition}, Advances in Difference Equations, 2013 (2013),
  pp.~1--16.

\bibitem{a10}
{\sc A.~B{\"o}ttcher and B.~Silbermann}, {\em Introduction to large truncated
  Toeplitz matrices}, Springer Science \& Business Media, New York, 1999.

\bibitem{chen2007fourier}
{\sc C.-M. Chen, F.~Liu, I.~Turner, and V.~Anh}, {\em A fourier method for the
  fractional diffusion equation describing sub-diffusion}, Journal of
  Computational Physics, 227 (2007), pp.~886--897.

\bibitem{chen2015fast}
{\sc S.~Chen, F.~Liu, X.~Jiang, I.~Turner, and V.~Anh}, {\em A fast
  semi-implicit difference method for a nonlinear two-sided space-fractional
  diffusion equation with variable diffusivity coefficients}, Applied
  Mathematics and Computation, 257 (2015), pp.~591--601.

\bibitem{cheng2015eulerian}
{\sc A.~Cheng, H.~Wang, and K.~Wang}, {\em A eulerian--lagrangian control
  volume method for solute transport with anomalous diffusion}, Numerical
  Methods for Partial Differential Equations, 31 (2015), pp.~253--267.

\bibitem{a9}
{\sc R.~M. Gray}, {\em Toeplitz and circulant matrices: A review}, vol.~77, Now
  Publishers, 2006.

\bibitem{hernandez2016solution}
{\sc M.~Hern{\'a}ndez-Hern{\'a}ndez, V.~N. Kolokoltsov, et~al.}, {\em On the
  solution of two-sided fractional ordinary differential equations of caputo
  type}, Fractional Calculus and Applied Analysis, 19 (2016), pp.~1393--1413.

\bibitem{huang2013two}
{\sc J.~Huang, Y.~Tang, L.~V{\'a}zquez, and J.~Yang}, {\em Two finite
  difference schemes for time fractional diffusion-wave equation}, Numerical
  Algorithms, 64 (2013), pp.~707--720.

\bibitem{jiang2015fast}
{\sc S.~Jiang, J.~Zhang, Q.~Zhang, and Z.~Zhang}, {\em Fast evaluation of the
  caputo fractional derivative and its applications to fractional diffusion
  equations}, arXiv preprint arXiv:1511.03453,  (2015).

\bibitem{jiang2011high}
{\sc Y.~Jiang and J.~Ma}, {\em High-order finite element methods for
  time-fractional partial differential equations}, Journal of Computational and
  Applied Mathematics, 235 (2011), pp.~3285--3290.

\bibitem{ke2015fast}
{\sc R.~Ke, M.~K. Ng, and H.-W. Sun}, {\em A fast direct method for block
  triangular toeplitz-like with tri-diagonal block systems from time-fractional
  partial differential equations}, Journal of Computational Physics, 303
  (2015), pp.~203--211.

\bibitem{keskin2011approximate}
{\sc Y.~Keskin, O.~Karao{\u{g}}lu, and S.~Servi}, {\em The approximate solution
  of high-order linear fractional differential equations with variable
  coefficients in terms of generalized taylor polynomials}, Mathematical and
  Computational Applications, 16 (2011), pp.~617--629.

\bibitem{li2015numerical}
{\sc C.~Li and F.~Zeng}, {\em Numerical methods for fractional calculus},
  vol.~24, CRC Press, 2015.

\bibitem{li2010finite}
{\sc W.~Li and X.~Da}, {\em Finite central difference/finite element
  approximations for parabolic integro-differential equations}, Computing, 90
  (2010), pp.~89--111.

\bibitem{lin2011finite}
{\sc Y.~Lin, X.~Li, and C.~Xu}, {\em Finite difference/spectral approximations
  for the fractional cable equation}, Mathematics of Computation, 80 (2011),
  pp.~1369--1396.

\bibitem{lin2007finite}
{\sc Y.~Lin and C.~Xu}, {\em Finite difference/spectral approximations for the
  time-fractional diffusion equation}, Journal of Computational Physics, 225
  (2007), pp.~1533--1552.

\bibitem{liu2014new1}
{\sc F.~Liu, P.~Zhuang, I.~Turner, K.~Burrage, and V.~Anh}, {\em A new
  fractional finite volume method for solving the fractional diffusion
  equation}, Applied Mathematical Modelling, 38 (2014), pp.~3871--3878.

\bibitem{liu2014rbf}
{\sc Q.~Liu, F.~Liu, I.~Turner, V.~Anh, and Y.~Gu}, {\em A rbf meshless
  approach for modeling a fractal mobile/immobile transport model}, Applied
  Mathematics and Computation, 226 (2014), pp.~336--347.

\bibitem{liu2015h}
{\sc Y.~Liu, Y.~Du, H.~Li, and J.~Wang}, {\em An h\^{} 1-galerkin mixed finite
  element method for time fractional reaction--diffusion equation}, Journal of
  Applied Mathematics and Computing, 47 (2015), pp.~103--117.

\bibitem{liu2014mixed}
{\sc Y.~Liu, Z.~Fang, H.~Li, and S.~He}, {\em A mixed finite element method for
  a time-fractional fourth-order partial differential equation}, Applied
  Mathematics and Computation, 243 (2014), pp.~703--717.

\bibitem{liu2016parallel}
{\sc Z.~Liu and X.~Li}, {\em A parallel cgs block-centered finite difference
  method for a nonlinear time-fractional parabolic equation}, Computer Methods
  in Applied Mechanics and Engineering, 308 (2016), pp.~330--348.

\bibitem{liu2017crank}
{\sc Z.~Liu and X.~Li}, {\em A crank--nicolson difference scheme for the time
  variable fractional mobile--immobile advection--dispersion equation}, Journal
  of Applied Mathematics and Computing,  (2017), pp.~1--20,
  \url{http://dx.doi.org/10.1007/s12190-016-1079-7}.

\bibitem{liu2017fast}
{\sc Z.~Liu and X.~Li}, {\em A fast finite difference method for a continuous
  static linear bond-based peridynamics model of mechanics}, Journal of
  Scientific Computing,  (2017), pp.~1--15.

\bibitem{saad1986gmres}
{\sc Y.~Saad and M.~H. Schultz}, {\em Gmres: A generalized minimal residual
  algorithm for solving nonsymmetric linear systems}, SIAM Journal on
  scientific and statistical computing, 7 (1986), pp.~856--869.

\bibitem{sousa2009finite}
{\sc E.~Sousa}, {\em Finite difference approximations for a fractional
  advection diffusion problem}, Journal of Computational Physics, 228 (2009),
  pp.~4038--4054.

\bibitem{sousa2012second}
{\sc E.~Sousa}, {\em A second order explicit finite difference method for the
  fractional advection diffusion equation}, Computers \& Mathematics with
  Applications, 64 (2012), pp.~3141--3152.

\bibitem{sousa2014explicit}
{\sc E.~Sousa}, {\em An explicit high order method for fractional advection
  diffusion equations}, Journal of Computational Physics, 278 (2014),
  pp.~257--274.

\bibitem{sun2006fully}
{\sc Z.-z. Sun and X.~Wu}, {\em A fully discrete difference scheme for a
  diffusion-wave system}, Applied Numerical Mathematics, 56 (2006),
  pp.~193--209.

\bibitem{wang2012fast}
{\sc H.~Wang and T.~S. Basu}, {\em A fast finite difference method for
  two-dimensional space-fractional diffusion equations}, SIAM Journal on
  Scientific Computing, 34 (2012), pp.~A2444--A2458.

\bibitem{wang2017preconditioned}
{\sc H.~Wang et~al.}, {\em A preconditioned fast finite difference method for
  space-time fractional partial differential equations}, Fractional Calculus
  and Applied Analysis, 20 (2017), pp.~88--116.

\bibitem{a2}
{\sc H.~Wang and H.~Tian}, {\em A fast galerkin method with efficient matrix
  assembly and storage for a peridynamic model}, Journal of Computational
  Physics, 231 (2012), pp.~7730--7738.

\bibitem{aa2}
{\sc H.~Wang and H.~Tian}, {\em A fast and faithful collocation method with
  efficient matrix assembly for a two-dimensional nonlocal diffusion model},
  Computer Methods in Applied Mechanics and Engineering, 273 (2014),
  pp.~19--36.

\bibitem{wang2011fast}
{\sc K.~Wang and H.~Wang}, {\em A fast characteristic finite difference method
  for fractional advection--diffusion equations}, Advances in water resources,
  34 (2011), pp.~810--816.

\bibitem{jin2010preconditioning}
{\sc J.~Xiao-qing}, {\em Preconditioning Techniques for Teoplitz Systems},
  Higher Education Press, Beijing, 2010.

\bibitem{zeng2013use}
{\sc F.~Zeng, C.~Li, F.~Liu, and I.~Turner}, {\em The use of finite
  difference/element approaches for solving the time-fractional subdiffusion
  equation}, SIAM Journal on Scientific Computing, 35 (2013), pp.~A2976--A3000.

\bibitem{zhang2012finite}
{\sc N.~Zhang, W.~Deng, and Y.~Wu}, {\em Finite difference/element method for a
  two-dimensional modified fractional diffusion equation}, Adv. Appl. Math.
  Mech, 4 (2012), pp.~496--518.

\bibitem{zhao2015two}
{\sc Y.~Zhao, P.~Chen, W.~Bu, X.~Liu, and Y.~Tang}, {\em Two mixed finite
  element methods for time-fractional diffusion equations}, Journal of
  Scientific Computing,  (2015), pp.~1--22.

\end{thebibliography}
